\documentclass[sn-mathphys]{sn-jnl_ohneSpringer}

\usepackage{amssymb,amsmath,colonequals,mathtools,stmaryrd,etoolbox,graphicx,psfrag,subfigure,xcolor}
\usepackage[autostyle]{csquotes}
\usepackage[english]{babel}
\usepackage{nicefrac}

\jyear{2023}%

\theoremstyle{thmstyleone}%
\newtheorem{theorem}{Theorem}[section]
\newtheorem{proposition}[theorem]{Proposition}%

\newtheorem{assum}{Assumption}

\theoremstyle{thmstyletwo}%

\theoremstyle{thmstylethree}%
\newtheorem{definition}[theorem]{Definition}%

\usepackage[numbers,sort&compress]{natbib}

\newcommand{\opt}{\textsc{opt}}

\newcommand{\XE}{X_{\textnormal{E}}}

\newcommand{\YN}{Y_{\textnormal{N}}}
\newcommand{\YWN}{Y_{\textnormal{WN}}}
\newcommand{\XWE}{X_{\textnormal{WE}}}

\newcommand{\QQ}{\mathbb{Q}}
\newcommand{\NN}{\mathbb{N}}
\newcommand{\gap}{\textsc{Gap}}
\newcommand{\constrained}{\textsc{Constrained}}

\newcommand{\dualrestrict}{\textsc{DualRestrict}}

\newcommand{\edom}[1]{\preceq_{\varepsilon}^{#1}}

\begin{document}

\title[Multiobjective Approximation: How exact can you be?]{Approximating Multiobjective Optimization Problems: How exact can you be?}


\author[1]{\fnm{Cristina} \sur{Bazgan}}\email{bazgan@lamsade.dauphine.fr}

\author[2,3]{\fnm{Arne} \sur{Herzel}}\email{arne.herzel@hswt.de}

\author[2]{\fnm{Stefan} \sur{Ruzika}}\email{ruzika@mathematik.uni-kl.de}

\author*[3,4]{\fnm{Clemens} \sur{Thielen}}\email{clemens.thielen@tum.de}

\author[1]{\fnm{Daniel} \sur{Vanderpooten}}\email{daniel.vanderpooten@lamsade.dauphine.fr}

\affil[1]{\orgname{Universit\'e Paris-Dauphine, PSL Reseach University, CNRS}, \orgdiv{LAMSADE}, \postcode{75016} \city{Paris}, \country{France}}

\affil[2]{\orgdiv{Department of Mathematics}, \orgname{University of Kaiserslautern-Landau}, \orgaddress{\street{Paul-Ehrlich-Str.~14}, \postcode{67663} \city{Kaiserslautern}, \country{Germany}}}

\affil[3]{\orgdiv{TUM Campus Straubing for Biotechnology and Sustainability}, \orgname{Weihenstephan-Triesdorf University of Applied Sciences}, \orgaddress{\street{Am~Essigberg~3}, \postcode{94315} \city{Straubing}, \country{Germany}}}

\affil[4]{\orgdiv{Department of Mathematics}, \orgname{Technical University of Munich}, \orgaddress{\street{Boltzmannstr.~3}, \postcode{85748} \city{Garching}, \country{Germany}}}


\abstract{It is well known that, under very weak assumptions, multiobjective optimization problems admit $(1+\varepsilon,\dots,1+\varepsilon)$-approximation sets (also called \emph{$\varepsilon$-Pareto sets}) of polynomial cardinality (in the size of the instance and in~$\frac{1}{\varepsilon}$). While an approximation guarantee of $1+\varepsilon$ for any $\varepsilon>0$ is the best one can expect for singleobjective problems (apart from solving the problem to optimality), even better approximation guarantees than $(1+\varepsilon,\dots,1+\varepsilon)$ can be considered in the multiobjective case since the approximation might be exact in some of the objectives. 

Hence, in this paper, we consider \emph{partially exact approximation sets} that require to approximate each feasible solution exactly, i.e., with an approximation guarantee of~$1$, in some of the objectives while still obtaining a guarantee of $1+\varepsilon$ in all others. We characterize the types of polynomial-cardinality, partially exact approximation sets that are guaranteed to exist for general multiobjective optimization problems. Moreover, we study minimum-cardinality partially exact approximation sets concerning (weak) efficiency of the contained solutions and relate their cardinalities to the minimum cardinality of a $(1+\varepsilon,\dots,1+\varepsilon)$-approximation set.}




\maketitle
\section{Introduction}

Many real-world optimization problems require taking into account multiple conflicting objective functions. Since solutions that optimize all objectives simultaneously do usually not exist, the solutions of interest are the so-called \emph{efficient} (or \emph{Pareto optimal}) solutions. A solution is called efficient if any other solution that is better in some objective is worse in at least one other objective. The images of the efficient solutions in the objective space are called \emph{nondominated points}. A main goal in multiobjective optimization is to determine the set of all nondominated points (the \emph{nondominated set}) and, for each of them, one corresponding efficient solution. One major difficulty, however, is \emph{intractability}, that is, the fact that the nondominated set (and the efficient set) may be exponentially large for discrete problems (see, e.g.,~\cite{Ehrgott:hard-to-say}), and typically infinite for continuous problems.

This provides a strong motivation to consider \emph{approximations} of multiobjective optimization problems. Approximations for general multiobjective problems have been investigated since the 1990s~\cite{Safer:PHD,Safer+Orlin:combinatorial}. Their existence and cardinality for general multiobjective problems have been systematically investigated in the seminal work of Papadimitriou and Yannakakis~\cite{Papadimitriou+Yannakakis:multicrit-approx}. They show that, for any instance of a multiobjective optimization problem with a constant number of positive-valued, polynomial-time computable objective functions and any $\varepsilon>0$, there exists a $(1+\varepsilon,\dots,1+\varepsilon)$-approximation set (also called a $(1+\varepsilon)$-approximation set or an \emph{$\varepsilon$-Pareto set}) whose cardinality is \emph{fully} polynomial, i.e., polynomial in the encoding length of the instance and~$\frac{1}{\varepsilon}$. Additionally, they prove that a $(1+\varepsilon)$-approximation set can be computed in (fully) polynomial time for every~$\varepsilon>0$ if and only if a certain auxiliary problem~$\gap_\delta$ (the \emph{gap problem}) can be solved in (fully) polynomial time for every~$\delta>0$.

It has been observed, however, that $(1+\varepsilon)$-approximation sets are far from unique and different $(1+\varepsilon)$-approximation sets can differ significantly in their cardinality. 
Consequently, further work has focused on the computation of $(1+\varepsilon)$-approximation sets whose cardinality can be bounded relative to the cardinality of a smallest (minimum-cardinality) $(1+\varepsilon)$-approximation set for the given instance~\cite{Bazgan+etal:min-pareto,Diakonikolas+Yannakakis:approx-pareto-sets,Koltun+Papadimitriou:approx-dom-repr,Vassilvitskii+Yannakakis:trade-off-curves}.

\smallskip

While an approximation guarantee of $1+\varepsilon$ for any $\varepsilon>0$ is the best solution quality one can expect for singleobjective problems (apart from solving the problem to optimality), even better approximation guarantees than $(1+\varepsilon,\dots,1+\varepsilon)$ can be considered in the multiobjective case since the approximation might be \emph{partially exact}, i.e., exact in some of the objectives (and $(1+\varepsilon)$-approximate in the others).
In fact, it has recently been shown in~\cite{HBRTV21} that, under similar assumptions as in~\cite{Papadimitriou+Yannakakis:multicrit-approx}, there exist polynomial-cardinality \emph{one-exact $(1+\varepsilon)$-approximation sets} that achieve an approximation guarantee of~$1$ in one of the objective functions (without loss of generality the first one) and $1+\varepsilon$ in the others. 

\smallskip

In this paper, we systematically investigate which types of partially exact approximation sets of polynomial cardinality are guaranteed to exist for general multiobjective problems. In particular, we show that, for any constant number~$p$ of objective functions and any $k\leq\lceil\frac{p}{2}\rceil$, polynomial-cardinality $(1+\varepsilon)$-approximation sets exist that approximate each feasible solution exactly in~$k$ of the objectives if we allow the objectives in which the approximation is exact to differ between different feasible solutions. We call these \emph{quasi-$k$-exact $(1+\varepsilon)$-approximation sets}. Their existence contrasts the fact that the efficient and nondominated sets are already intractable for many biobjective problems, which implies that polynomial-cardinality $(1+\varepsilon)$-approximation sets that are exact in the same objectives for all solutions can in general be exact in at most one objective~\cite{HBRTV21}. We obtain our general result by providing a general existence proof that generalizes the existence proofs for $(1+\varepsilon)$-approximation sets and one-exact $(1+\varepsilon)$-approximation sets provided in~\cite{Papadimitriou+Yannakakis:multicrit-approx} and~\cite{HBRTV21}, respectively.

Moreover, we investigate the question of (weak) efficiency of the solutions contained in minimum-cardinality partially exact $(1+\varepsilon)$-approximation sets of different types and relate their cardinalities 
to the minimum cardinality of a $(1+\varepsilon)$-approximation set for a given instance. In particular, we show that, for every instance, the minimum cardinality of a quasi-$1$-exact $(1+\varepsilon)$-approximation set equals the minimum cardinality of an ordinary $(1+\varepsilon)$-approximation set\footnote{We sometimes use the term \emph{ordinary} $(1+\varepsilon)$-approximation set to refer to $(1+\varepsilon)$-approximation sets if we want to explicitly distinguish them from the more demanding partially exact $(1+\varepsilon)$-approximation sets.}. This contrasts the result from~\cite{HBRTV21} stating that the minimum cardinality of a one-exact $(1+\varepsilon)$-approximation set can be an arbitrary factor larger than the minimum cardinality of an ordinary $(1+\varepsilon)$-approximation set. By generalizing the corresponding proof from~\cite{HBRTV21}, however, we show that, for $k\geq 2$, the cardinality of a quasi-$k$-exact $(1+\varepsilon)$-approximation set can again be an arbitrary factor larger than the minimum cardinality of an ordinary $(1+\varepsilon)$-approximation set.

Finally, we present several results concerning the polynomial-time computability of partially exact approximation sets. In particular, we show that the (fully) polynomial-time solvability of the gap problem does not suffice for the computation of any type of partially exact approximation set in general.

\smallskip

While we focus on results that are applicable to general multiobjective optimization problems under weak assumptions, there also exists a large body of work on approximations for specific multiobjective problems. For an up-to-date survey on both general and problem-specific approximation results, we refer to~\cite{Herzel+etal:survey}.

\section{Preliminaries}

Multiobjective optimization problems can contain objective functions that are to be minimized or maximized (or even a combination of both). However, we only consider the case of minimization problems in this article, i.e., all objectives are to be minimized. This is without loss of generality here since all our arguments can be straightforwardly adapted to the case where some or all objective functions are to be maximized.

\begin{definition}[Multiobjective Optimization Problem]
	A \emph{multiobjective optimization problem}~$\Pi$ is given by a set of instances. Each instance~$I$ consists of a (finite or infinite) set~$X^I$ of feasible solutions and a vector~$f^I = (f^I_1,\ldots, f^I_{p})$ of $p$~objective functions~$f^I_i : X^I \to \QQ$ for $i = 1,\ldots,p$ to be minimized. The feasible set~$X^I$ might not be given explicitly. 
\end{definition}

Here, the number~$p$ of objective functions in a multiobjective optimization problem~$\Pi$ is assumed to be constant. Moreover, as is common in the context of approximation of multiobjective optimization problems, we assume that the objective functions take only positive, rational values and are polynomial-time computable (cf.~\cite{Papadimitriou+Yannakakis:multicrit-approx,Diakonikolas+Yannakakis:approx-pareto-sets,Vassilvitskii+Yannakakis:trade-off-curves,HBRTV21}). Additionally, we use the following standard assumption\footnote{Given the previously-stated assumptions on the objective functions, this general assumption also used in~\cite{Diakonikolas+Yannakakis:approx-pareto-sets,Vassilvitskii+Yannakakis:trade-off-curves,HBRTV21} holds for large classes of problems - in particular, for combinatorial problems. Together with the previously-stated assumptions, it guarantees the existence of all versions of approximation sets considered here and in~\cite{Papadimitriou+Yannakakis:multicrit-approx,Diakonikolas+Yannakakis:approx-pareto-sets,Vassilvitskii+Yannakakis:trade-off-curves,HBRTV21}. We note, however, that specific types of approximation sets can also exist for some classes of problems for which this assumption does not formally hold (e.g., for multiobjective linear programming problems, see~\cite{Papadimitriou+Yannakakis:multicrit-approx}).}:

\begin{assum}\label{assum:M}
    For any multiobjective optimization problem~$\Pi$, there exists a polynomial~$\mathcal{P}$ such that for any instance~$I$ of~$\Pi$, there exists a constant~$M^I \leq \mathcal{P}(\textnormal{enc}(I))$ such that, for any $x \in X^I$ and any $i \in \{1,\ldots, p\}$, we have 
  $\mbox{enc}(f^I_i(x)) \leq M^I$, where $\mbox{enc}(I)$ denotes the encoding 
   length of the instance~$I$ and $\mbox{enc}(f^I_i(x))$ 
    denotes the encoding length of the value~$f^I_i(x) \in \QQ_{>0}$ in binary.
  This, in particular, implies that, for any instance~$I$ and any objective function value~$f_i^I(x)$, we have $2^{-M^I} \leq f_i^I(x) \leq 2^{M^I}$. Also, any two values~$f_i^I(x)$ and $f_i^I(x')$ differ by at least~$2^{-2M^I}$ if they are not equal. 
\end{assum}
	
	
In the following, we sometimes blur the distinction between the problem~$\Pi$ and a concrete instance $I = (X^I,f^I)$ and we drop the superscript~$I$ indicating the dependence on the instance in~$X^I$, $f^I$, $M^I$, etc. whenever the instance is clear from the context.


\smallskip

For multiobjective optimization problems, the solutions of interest are the so-called \emph{efficient solutions} for which any other solution that is better in some objective is worse in at least one other objective:

\begin{definition}\label{def:dominance}
	For an instance~$I$ of a multiobjective optimization problem, a solution~$x \in X$ \emph{dominates} another solution~$x' \in X$ (denoted as $x\preceq x'$)
    if $f_i(x) \leq f_i(x')$ for $i = 1,\ldots, p$ and $f_i(x) < f_i(x')$ for at least one~$i$. Moreover, a solution~$x \in X$ \emph{strictly dominates} another solution~$x' \in X$ if $f_i(x) < f_i(x')$ for $i = 1,\ldots, p$. A solution~$x \in X$ is called \emph{(weakly) efficient} if it is not (strictly) dominated by any other solution~$x' \in X$. In this case, we call the corresponding image~$f(x) \in f(X) \subseteq \QQ^p$ a \emph{(weakly) nondominated point}. The set~$\XE\subseteq X$ of all efficient solutions is called the \emph{efficient set} (or \emph{Pareto set}) and the set~$\YN\colonequals f(\XE)$ of nondominated points is called the \emph{nondominated set}. Similarly, the set~$\XWE\subseteq X$ of all weakly efficient solutions is called the \emph{weakly efficient set} and the set~$\YWN\colonequals f(\XWE)$ of weakly nondominated points is called the \emph{weakly nondominated set}.
\end{definition}

 
One way to circumvent the problem that the nondominated set often has cardinality exponential in the input size (see, e.g.,~\cite{Ehrgott:hard-to-say}) is the concept of approximation. Here, instead of requiring to select at least one corresponding efficient solution for each nondominated point, the requirement is relaxed so that each image point is only required to be dominated ``approximately'' by some image of a selected solution. This idea is formalized as follows:

\begin{definition}\label{def:approximation}
Let~$(X,f)$ be an instance of a multiobjective optimization problem and let $\alpha=(\alpha_1,\dots,\alpha_p)\in\mathbb{R}^p$ with $\alpha_i\geq 1$ for all $i\in\{1,\dots,p\}$.
We say that a feasible solution $x\in X$ \emph{$\alpha$-approximates} another feasible solution $x'\in X$ if $f_i(x) \leq \alpha_i\cdot f_i(x')$ for $i\in\{1,\dots,p\}$.
\end{definition}


When fixing a desired vector~$\alpha$ of approximation guarantees (or a set of possible desired vectors), approximation as in Definition~\ref{def:approximation} can equivalently be expressed as a relation~$R \subseteq X\times X$, where $(x,x')\in R$ (also denoted as $x R x'$) if and only if~$x'$ is $\alpha$-approximated by~$x$ for some desired vector~$\alpha$. In the following, we consider general approximate dominance relations that are \emph{monotonic} in the following sense: 

\begin{definition}\label{def:dominance-compatibility}
Let~$(X,f)$ be an instance of a multiobjective optimization problem. A relation $R \subseteq X\times X$ is called \emph{monotonic} if $f_i(x)\leq f_i(x')$ for all $i\in\{1\ldots,p\}$ implies that $x R x'$.
\end{definition}

Note that any monotonic relation~$R$ is, in particular, reflexive, i.e., $x R x$ for all $x\in X$. We are particularly interested in the following monotonic approximate dominance relations, which give rise to ordinary $(1+\varepsilon)$-approximation sets as introduced in~\cite{Papadimitriou+Yannakakis:multicrit-approx} and the different types of partially exact approximation sets considered here:

\begin{definition}\label{def:relations}
Let~$(X,f)$ be an instance of a multiobjective optimization problem and let $\varepsilon>0$. We define the following relations on~$X\times X$:
\begin{itemize}
	\item \emph{$(1+\varepsilon)$-dominance}: $x \edom{} x' :\Leftrightarrow$ $x$ $\alpha$-approximates~$x'$, where $\alpha_i=1+\varepsilon$ for~$i=1,\ldots,p$,
    \item \emph{one-exact $(1+\varepsilon)$-dominance}: $x \edom{1} x' :\Leftrightarrow$ $x$ $\alpha$-approximates~$x'$, where $\alpha_1=1$ and $\alpha_i=1+\varepsilon$ for all $i\geq 2$,
	\item \emph{two-exact $(1+\varepsilon)$-dominance}: $x \edom{1,2} x' :\Leftrightarrow$ $x$ $\alpha$-approximates~$x'$, where $\alpha_1=\alpha_2=1$ and $\alpha_i=1+\varepsilon$ for all $i\geq 3$,
	\item \emph{quasi-$k$-exact $(1+\varepsilon)$-dominance}: $x \edom{(k)} x' :\Leftrightarrow$ $x$ $\alpha$-approximates~$x'$, where~$k$ components~$\alpha_i$ are equal to~$1$ and the other~$p-k$ are equal to $1+\varepsilon$.
    \item \emph{one-exact, quasi-$k$-exact $(1+\varepsilon)$-dominance}: $x \edom{1,(k)} x' :\Leftrightarrow$ $x$ $\alpha$-approximates~$x'$, where $\alpha_1=1$, $k-1$ of the other components~$\alpha_i$ are equal to~$1$, and the remaining~$p-k$ are equal to $1+\varepsilon$.\footnote{Note that this means that $\edom{1,(k)}\; =\; \edom{1} \cap  \edom{(k)}$ as subsets of~$X\times X$.}
\end{itemize}
The relations~$\edom{1}$, $\edom{1,2}$, $\edom{(k)}$, and $\edom{1,(k)}$ will also be referred to as \emph{partially exact} (approximate dominance) relations in the following.
\end{definition}



Given an approximate dominance relation, approximation sets for multiobjective optimization problems are defined as follows: 

\begin{definition}\label{def:R-approx}
    Let~$(X,f)$ be an instance of a multiobjective optimization problem and let~$R$ be a relation on $X\times X$.
	A set~$P\subseteq X$ of feasible solutions is called an \emph{$R$-approximation set} for~$(X,f)$ if, for any feasible solution~$x'\in X$, there exists a solution~$x\in P$ such that $x R x'$ (i.e., $x$ \emph{$R$-dominates}~$x'$).

    \smallskip

    For $\varepsilon>0$, $R$-approximation sets for the relations~$\edom{}$, $\edom{1}$, $\edom{1,2}$, $\edom{(k)}$, and $\edom{1,(k)}$ are referred to as \emph{(ordinary) $(1+\varepsilon)$-approximation sets}, \emph{one-exact $(1+\varepsilon)$-approximation sets}, \emph{two-exact $(1+\varepsilon)$-approximation sets}, \emph{quasi-$k$-exact $(1+\varepsilon)$-approximation sets}, and \emph{one-exact, quasi-$k$-exact $(1+\varepsilon)$-approximation sets}, respectively.

\end{definition}

In the following, we will often use that, since monotonic relations~$R$ are reflexive, $R$-approximation sets for monotonic relations correspond to \emph{$R$-dominating sets} for~$X$, i.e., to subsets~$P\subseteq X$ of feasible solutions such that any feasible solution~$x'\notin P$ is $R$-dominated by some solution~$x\in P$. This correspondence, in particular, holds for the monotonic relations $\edom{}$, $\edom{1}$, $\edom{1,2}$, $\edom{(k)}$, and $\edom{1,(k)}$.
 
\smallskip
 
We are interested in approximation sets whose cardinality is bounded by a polynomial in the encoding size of the given instance~$(X,f)$. For the relations~$\edom{}$, $\edom{1}$, $\edom{1,2}$, $\edom{(k)}$, and $\edom{1,(k)}$ that depend on a given value~$\varepsilon>0$, we are specifically interested in approximation sets whose cardinality is \emph{fully polynomial}, i.e., polynomial in the encoding size of the instance and in~$\frac{1}{\varepsilon}$.

\smallskip

In the literature, several auxiliary problems have been defined for the computation of (fully) polynomial-cardinality approximation sets~\cite{Papadimitriou+Yannakakis:multicrit-approx,Diakonikolas+Yannakakis:approx-pareto-sets,Vassilvitskii+Yannakakis:trade-off-curves,HBRTV21}, which are also used in the following. The first such problem, whose (fully) polynomial-time solvability has been shown to characterize the (fully) polynomial-time compatibility of $(1+\varepsilon)$-approximation sets in~\cite{Papadimitriou+Yannakakis:multicrit-approx} is the \emph{gap problem}:

\begin{definition}[Gap Problem]\label{def:gap}
Given an instance~$(X,f)$ of a $p$-objective optimization problem, a point~$b\in\QQ^p$, and some $\delta>0$, the problem $\gap_{\delta}(b)$ is the following: Either return a feasible solution~$x\in X$ with $f_i(x)\leq b_i$ for $i=1,\ldots,p$, or answer correctly that there does not exist any feasible solution~$x'$ such that $f_i(x')\leq \frac{b_i}{1+\delta}$ for all $i=1,\ldots,p$.
 \end{definition}

 The following auxiliary problem, which scalarizes the multiobjective problem via budget constraints on all but one objective function, is widely used both in practice and in the theoretical literature on multiobjective optimization:

\begin{definition}[Budget-Constrained Problem]
		Given an instance~$(X,f)$ of a multiobjective optimization problem and bounds $b_i>0$, $i = 2,\ldots,p$, for all objective functions except the first one, the problem $\constrained(b_2\ldots,b_p)$ is the following: Either answer that there does not exist a feasible solution $x' \in X$ with $f_i(x') \leq b_i$ for $i=1,\ldots,p$, or return a feasible solution that minimizes~$f_1$ among all such solutions, i.e., return $x \in X$ with \begin{align*}f_1(x) &=  \opt_1(b_2,\ldots,b_p) \colonequals  \min_{x'\in X} \{f_1(x'):\,f_i(x') \leq b_i \textnormal{ for } i = 2,\ldots,p\}, \text{ and}\\
		f_i(x) &\leq  b_i, \qquad  i = 2,\ldots, p.
		\end{align*}
	\end{definition}
	
	Note that $\constrained$ is hard to solve even for the biobjective version of most relevant problems such as shortest path. A way to circumvent the hardness of $\constrained$ is to consider solutions that violate the given bounds slightly, while requiring an objective value that is at least as good as the objective value of any solution that respects the bounds~\cite{Diakonikolas+Yannakakis:approx-pareto-sets}:
 
	
	
	
	\begin{definition}[{\normalfont$\dualrestrict$}]
		Given an instance~$(X,f)$ of a multiobjective optimization problem, bounds $b_i>0$, $i = 2,\ldots,p$, for all objective functions except the first one, and some~$\delta>0$, the problem $\dualrestrict_{\delta}(b_2,\ldots, b_p)$ is the following: Either answer that there does not exist a feasible solution $x' \in X$ with $f_i(x') \leq b_i$ for $i=1,\ldots,p$, or return $x \in X$ with 
		\begin{align*}f_1(x) &\leq \phantom{(1+\delta) \cdot} \opt_1(b_2,\ldots,b_p)&\\
		f_i(x) &\leq (1+\delta) \cdot b_i,&&   i = 2,\ldots, p.\end{align*}	
	\end{definition}
	

    Note that $\constrained$ can be viewed as the limit case where $\delta = 0$ in $\dualrestrict$. 

    \smallskip
	
	The auxiliary problems $\constrained$ and $\dualrestrict$ can also be defined such that, instead of the first one, some other objective is to be optimized subject to budgets on the remaining objectives. In the following, we sometimes indicate which objective is optimized by an upper index. For example, $\dualrestrict^i$ denotes the $\dualrestrict$ problem with a bound on all objectives but the $i$-th one.

 

 \section{Existence of Approximation Sets}
 
 
 We start by investigating the existence of $R$-approximations of (fully) polynomial cardinality under weak assumptions on the approximate dominance relation~$R$. For the $(1+\varepsilon)$-dominance relation~$R = \,\edom{}$, this is established in the seminal work of Papadimitriou and Yannakakis~\cite{Papadimitriou+Yannakakis:multicrit-approx}, who show the existence of $(1+\varepsilon)$-approximations of fully polynomial cardinality $\mathcal{O}((\frac M \varepsilon)^{p-1})$. 
 
 
We now generalize the existence proof from~\cite{Papadimitriou+Yannakakis:multicrit-approx} to more demanding approximate dominance relations~$R$ that refine~$\preceq_{\varepsilon}$, which, as shown afterwards in Theorem~\ref{thm:specific-existence}, include the partially exact approximate dominance relations~$\edom{1}$ and~$\edom{(k)}$. 
 
 \begin{theorem}\label{thm:generalexistence}
    Let~$(X,f)$ be an instance of a $p$-objective optimization problem and let $\varepsilon >0$. Consider a monotonic relation~$R'$ defined on $X\times X$ and let $R\colonequals\edom{} \cap\, R'$. If there exists a constant~$c$ (independent of the instance size and~$\frac{1}{\epsilon}$) such that every nonempty subset $X'\subseteq X$ admits an $R'$-dominating set~$P_{X'}\subseteq X'$ of cardinality at most~$c$, then there exists an $R$-approximation set of cardinality $\mathcal{O}\bigl(\bigl(\frac M \varepsilon\bigr)^{p-1}\bigr)$. 
\end{theorem}

\begin{proof}
Consider the hyperrectangle $[2^{-M},2^M] \times \ldots \times [2^{-M},2^M]$, in which all images of feasible solutions are contained. We create a grid by subdividing this hyperrectangle into smaller hyperrectangles such that, in each dimension, the ratio of the larger to the smaller coordinate is $1 + \varepsilon$. Thus, the total number of subdivisions in each dimension is $\left\lceil\log_{1+\varepsilon}\left(\frac{2^M}{2^{-M}}\right)\right\rceil=\left\lceil\log_{1+\varepsilon}\left(2^{2M}\right)\right\rceil$. 
By the assumption on~$R'$, any subset~$X'$ of feasible solutions admits an $R'$-dominating set~$P_{X'}\subseteq X'$ of cardinality at most~$c$. Moreover, any two solutions~$x$ and~$x'$ with images in the same hyperrectangle of the grid satisfy $x \preceq_{\varepsilon} x'$ by construction of the grid. Thus, if~$X'$ is the preimage of a nonempty hyperrectangle of the grid, then the $R'$-dominating set~$P_{X'}$ is actually an $R$-dominating set for~$X'$. Consequently, forming the union of the sets~$P_{X'}$ over all subsets~$X'$ that are preimages of nonempty hyperrectangles of the grid yields an $R$-approximation set. Moreover, since~$R$ is monotonic, it suffices to consider only weakly nondominated nonempty hyperrectangles.

To bound the cardinality of this $R$-approximation set, we note that, for each hyperrectangle $[l_1,u_1]\times\cdots\times[l_p,u_p]$ of the grid, we have $u_i = (1+\varepsilon)\cdot l_i$ for all $i\in\{1,\ldots,p\}$, so the vectors~$l=(l_1,\ldots,l_p)$ and~$u=(u_1,\ldots,u_p)$ of lower and upper bounds of the hyperrectangle lie on a straight line through the origin. Calling the set of all hyperrectangles of the grid whose bounds lie on the same straight line through the origin a \emph{diagonal}, we see that (1) at most one hyperrectangle on each diagonal is weakly nondominated and nonempty, and (2) the number of diagonals is in $\mathcal{O}\bigl(p\cdot\bigl\lceil\log_{1+\varepsilon}\bigl(2^{2M}\bigr)\bigl\rceil^{p-1}\bigr)=\mathcal{O}\bigl(\bigl(\frac M \varepsilon\bigr)^{p-1}\bigr)$. Hence, since the cardinality of the constructed $R$-approximation set is at most $c$~times the number of  weakly nondominated nonempty hyperrectangles and~$c$ is a constant, its cardinality is also in $\mathcal{O}\bigl(\bigl(\frac M \varepsilon\bigr)^{p-1}\bigr)$.
 \end{proof}

We remark that the proof of Theorem~\ref{thm:generalexistence} also yields an $R$-approximation set of (fully) polynomial cardinality if the cardinality of the $R'$-dominating set~$P_{X'}$ for each subset~$X'\subseteq X$ is only required to be polynomial in the instance size (and in~$\frac{1}{\varepsilon}$) instead of constant. In this case, the cardinality of the obtained $R$-approximation set would be upper bounded by $\mathcal{O}\bigl(\bigl(\frac M \varepsilon\bigr)^{p-1}\bigr)$ times the worst-case cardinality of~$P_{X'}$ taken over all preimages~$X'$ of weakly nondominated nonempty hyperrectangles of the grid.

\smallskip

The following theorem shows that the general result in Theorem~\ref{thm:generalexistence}, in particular, yields existence results for $(1+\varepsilon)$-approximation sets, one-exact $(1+\varepsilon)$-approximation sets, and quasi-$k$-exact $(1+\varepsilon)$-approximation sets of fully polynomial cardinality. Here the existence results for $(1+\varepsilon)$-approximation sets and for one-exact $(1+\varepsilon)$-approximation sets, which are already known from~\cite{Papadimitriou+Yannakakis:multicrit-approx} and~\cite{HBRTV21}, respectively, follow immediately from the theorem. The result concerning quasi-$k$-exact $(1+\varepsilon)$-approximation sets, however, has not been shown in the literature before and requires a more involved proof, which is illustrated in Figure~\ref{fig:existence-quasiexact}.

\begin{theorem}\label{thm:specific-existence}
 For any instance $(X,f)$ of a $p$-objective optimization problem and any $\varepsilon >0$,  there exists an $R$-approximation set of cardinality $\mathcal{O}((\frac M \varepsilon)^{p-1})$
 \begin{enumerate}
     \item[1)] for $R = \,\edom{}$ ($(1+\varepsilon)$-approximation set),
     \item[2)] for $R = \,\edom{1}$ (one-exact $(1+\varepsilon)$-approximation set),
     \item[3)] for $R = \,\edom{(k)}$ for any $k\leq \lceil \frac{p}{2}\rceil$ (quasi-$k$-exact $(1+\varepsilon)$-approximation set).
 \end{enumerate}
 \end{theorem}
 
 \begin{proof}\hfill
 \begin{enumerate}
     \item[1)] Choosing $R' \colonequals X \times X$ in Theorem~\ref{thm:generalexistence} yields $R= \edom{}$. Moreover, $R'$ is clearly monotonic and, for any nonempty subset~$X'\subseteq X$, any solution~$x\in X'$ constitutes an $R'$-dominating set, so the result follows from Theorem~\ref{thm:generalexistence}.
     \item[2)] Choosing $R'\colonequals \{(x,x') \in X \times X: f_1(x)\leq f_1(x')\}$ in Theorem~\ref{thm:generalexistence} yields $R= \edom{1}$. Again, $R'$ is clearly monotonic. Moreover, given any nonempty subset~$X'\subseteq X$, any solution~$x \in X'$ such that $f_1(x) = \min_{x'\in X'} f_1(x')$ constitutes an $R'$-dominating set (note that, by Assumption~\ref{assum:M}, the number of possible images is finite, so the minimum is guaranteed to exist).
     \item[3)] Given $k\leq \lceil \frac{p}{2}\rceil$, choosing
     \begin{align*}
         R'
         \colonequals \{(x,x') \in X \times X: f_j(x)\leq f_j(x')\mbox{ for at least } k \mbox{ objectives}~f_j\}, 
     \end{align*}     
     in Theorem~\ref{thm:generalexistence} yields $R = \,\edom{(k)}$. 
    Since $R'$ is clearly monotonic, it remains to show that there exists an $R'$-dominating set of constant cardinality for any nonempty subset~$X'\subseteq X$. To show this, we use a result from~\cite{ABKKW06} about dominating sets in \emph{$k$-majority tournaments}. A $k$-majority tournament is a directed graph on a finite vertex set~$V$ where, given $2k-1$ linear orders of~$V$, there exists a directed arc from~$u$ to~$v$ if and only if~$u$ is ranked higher than~$v$ in at least~$k$ of the orders. The result from~\cite{ABKKW06} states that any $k$-majority tournament admits a dominating set of cardinality $\mathcal{O}(k\log k)$.

     In our situation, given some nonempty subset~$X'\subseteq X$, each objective~$f_j$ defines a reflexive, transitive, and strongly connected order on~$f(X')$ by ranking~$f(x)$ higher than~$f(x')$ if and only if $f_j(x)\leq f_j(x')$. This order can be made antisymmetric (i.e., turned into a linear order) by breaking ties within any subset of elements that have the same $f_j$-value via an arbitrary linear order on this subset. Since $k \leq  \lceil \frac{p}{2}\rceil$, this construction yields $p\geq 2k-1$ linear orders on~$f(X')$. Moreover, by Assumption~\ref{assum:M}, the number of possible images the objective space is finite, so $f(X')$ is a finite set. Consequently, the directed graph on~$f(X')$ in which there exists a directed arc from~$f(x)$ to~$f(x')$ if and only if~$f(x)$ is ranked higher than~$f(x')$ in at least~$k$ of the orders is finite and contains a $k$-majority tournament on~$f(X')$ (since, in case that $p>2k-1$, the additional $p-2k+1$~linear orders only lead to additional arcs in the graph). Consequently, by the result of~\cite{ABKKW06}, the graph contains a dominating set of constant cardinality $\mathcal{O}(k\log k)$. By construction of the graph, choosing one corresponding solution~$x$ for each image in this dominating set yields an $R'$-dominating set for~$X'$.\qedhere
 \end{enumerate}
 \end{proof}
\begin{figure}[ht!]
	\begin{center}
		\includegraphics[width = \textwidth]{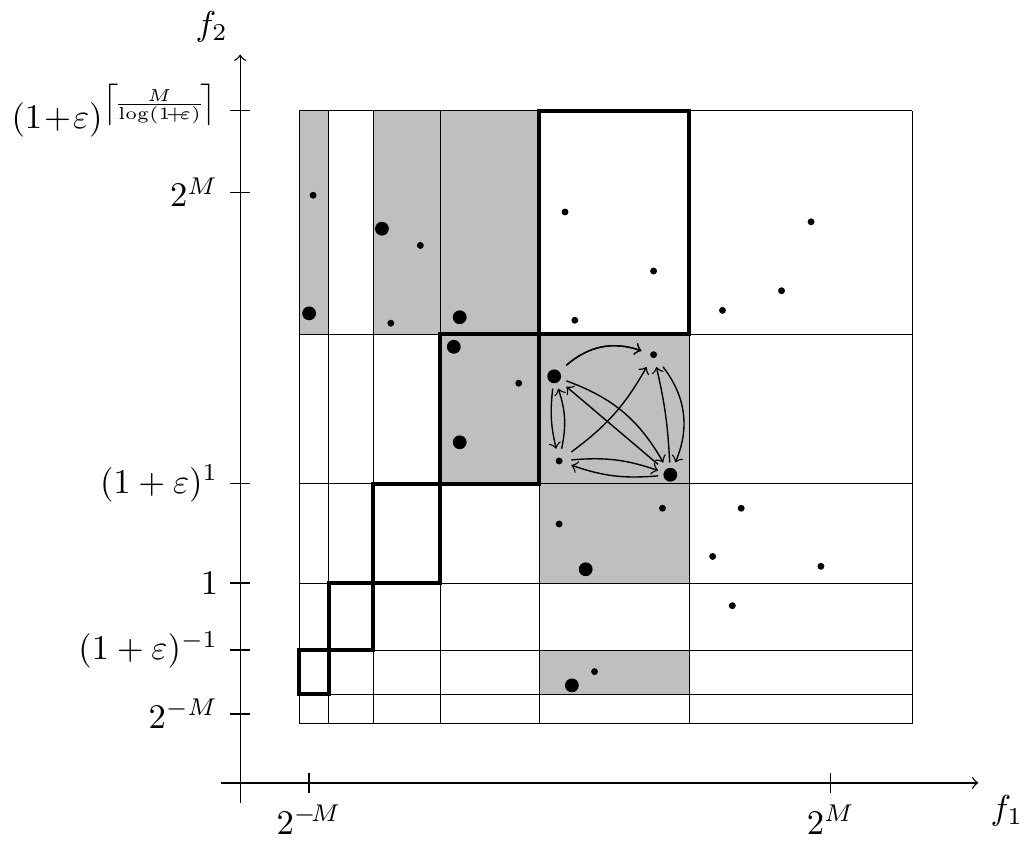}
		\caption[Illustration of the proof of Theorem~\ref{thm:existence_disjapprox}]{\textbf{Existence of quasi-$k$-exact $(1+\varepsilon)$-approximation sets of fully polynomial
 cardinality.} For each weakly nondominated nonempty hyperrectangle in the objective space as in the proof of Theorem~\ref{thm:generalexistence}, a dominating set for the supergraph of a $k$-majority tournament defined on the images in the hyperrectangle is selected to obtain a quasi-$k$-exact $(1+\varepsilon)$-approximation set. There exists at most one weakly nondominated nonempty hyperrectangle in each diagonal of the grid. A quasi-$k$-exact $(1+\varepsilon)$-approximation set is given by preimages of the bold points. Weakly nondominated nonempty hyperrectangles are indicated by shaded boxes. One diagonal is indicated by the bold lines. The constructed supergraph of a $k$-majority tournament is illustrated for one hyperrectangle.}\label{fig:existence-quasiexact}
	\end{center}
\end{figure}



In the remainder of this section, we show that the existence results for one-exact $(1+\varepsilon)$-approximation sets and quasi-$k$-exact $(1+\varepsilon)$-approximation sets with $k=\lceil \frac{p}{2}\rceil$ from Theorem~\ref{thm:specific-existence} are tight in the sense that polynomial-cardinality $(1+\varepsilon)$-approximation sets with additional exact components do not exist in general. This is made precise for the case of one-exact $(1+\varepsilon)$-approximation sets in the following theorem.


 
 \begin{theorem}\label{thm:nonexistance1}
 There exist instances of many $p$-objective combinatorial optimization problems, including the multiobjective versions of shortest path, minimum spanning tree, assignment, minimum $s$-$t$-cut, knapsack, and traveling salesman, that do not admit a polynomial-cardinality one-exact, quasi-$2$-exact $(1+\varepsilon)$-approximation set for any $\varepsilon>0$.
  \end{theorem}
\begin{proof}
The biobjective versions of all of these problems are known to be intractable, i.e., there exist instances for which the cardinality of the nondominated set is exponential in the instance size. From such an intractable biobjective instance, construct a $p$-objective instance where objective~$f_1$ corresponds to the first objective function of the biobjective instance and the other $p-1$~objective functions correspond to the second objective function of the biobjective instance. Then, at least one solution corresponding to each nondominated point of the biobjective instance is required in any one-exact, quasi-$2$-exact $(1+\varepsilon)$-approximation set of the $p$-objective instance.
\end{proof}


The above theorem shows that, for many important problems, one-exact $(1+\varepsilon)$-approximation sets are the best polynomial-cardinality partially exact approximation sets one can hope for when requiring to be exact in one fixed objective.
In fact, we are not aware of many (non-trivial) problems that admit polynomial-cardinality one-exact, quasi-$2$-exact $(1+\varepsilon)$-approximation sets. One notable exception is the biobjective minimum cut problem, for which it is known that the cardinality of the efficient set is polynomial in the instance size~\cite{Aissi+etal:MinCut}. Since the efficient set corresponds to a two-exact $(1+\varepsilon)$-approximation set in the biobjective case, this in particular shows the existence of a polynomial-cardinality one-exact, quasi-$2$-exact $(1+\varepsilon)$-approximation set. 


\smallskip

We now consider quasi-$k$-exact $(1+\varepsilon)$-approximation sets, that is, approximation sets that approximate each feasible solution exactly in $k$~objectives (and with a factor of $(1+\varepsilon)$-approximate in all other objectives), but allow the objectives in which the approximation is exact to differ depending on the approximated solution. Here, a similar argument as in the proof of Theorem~\ref{thm:nonexistance1} shows that, for many multiobjective combinatorial optimization problems, polynomial-cardinality quasi-$k$-exact $(1+\varepsilon)$-approximation sets for $k >  \lceil \frac{p}{2}\rceil$ are not guaranteed to exist for all instances.

\begin{theorem}\label{thm:nonexistance2}
There exist instances of many $p$-objective combinatorial optimization problems, including the multiobjective versions of shortest path, minimum spanning tree, assignment, minimum $s$-$t$-cut, knapsack, and traveling salesman, that do not admit a polynomial-cardinality quasi-$k$-exact $(1+\varepsilon)$-approximation set for any $k>\lceil \frac{p}{2}\rceil$ and any $\varepsilon>0$.
\end{theorem}
\begin{proof}
 As noted in the proof of Theorem~\ref{thm:nonexistance1}, the biobjective versions of all of these problems are known to be intractable. From such an intractable biobjective instance, construct a $p$-objective instance where the first $\lceil\frac{p}{2}\rceil$ objectives corresponds to the first objective of the biobjective instance and the other $\lfloor\frac{p}{2}\rfloor$ objectives correspond to the second objective function of the biobjective instance. Then, for $k>\lceil \frac{p}{2}\rceil$, at least one solution corresponding to each nondominated point of the biobjective instance is required in any quasi-$k$-exact $(1+\varepsilon)$-approximation set of the $p$-objective instance.\qedhere
\end{proof}

Figure~\ref{fig:tractability-boundary} summarizes the known tractability and intractability results for different types of approximation sets.

\begin{figure}[ht!]
	\begin{center}
        \includegraphics[width = \textwidth]{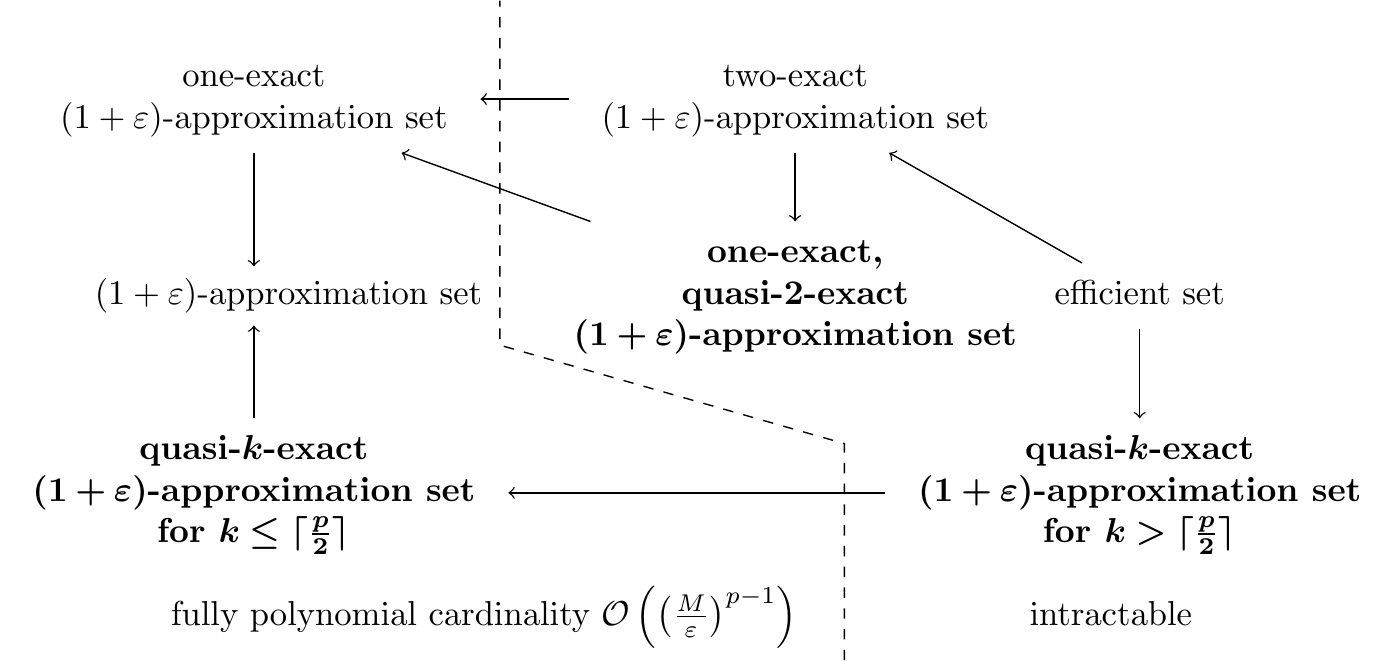}
		\caption[Tractability and intractability results for different types of approximation sets]{Tractability and intractability results for different types of approximation sets. An arrow indicates that one type of approximation set fulfills the requirements of the other. The dashed line marks the boundary between tractability and intractability in general $p$-objective optimization problems. New results obtained in this paper are shown in bold. The result on $(1+\varepsilon)$-approximation sets is shown in~\cite{Papadimitriou+Yannakakis:complexity}, and the results on one-exact and two-exact $(1+\varepsilon)$-approximation sets are shown in~\cite{HBRTV21}.}\label{fig:tractability-boundary}
	\end{center}
\end{figure}

\section{Minimum-Cardinality Approximation Sets}\label{sec:min-size}

Theorem~\ref{thm:specific-existence} implies that an asymptotic upper bound for the minimum cardinality of an ordinary, a one-exact, and a quasi-$k$-exact $(1+\varepsilon)$-approximation set with $k\leq \lceil \frac{p}{2}\rceil$ is in $\mathcal{O}\big(\!\left(\frac M \varepsilon\right)^{p-1}\big)$ and it is easy to see that this bound is tight, i.e., there exist instances where even a minimum-cardinality ordinary $(1+\varepsilon)$-approximation set consists of $\Theta\big(\!\left(\frac M \varepsilon\right)^{p-1}\big)$ solutions. For a specific instance, however, as pointed out in~\cite{Vassilvitskii+Yannakakis:trade-off-curves,Diakonikolas+Yannakakis:approx-pareto-sets}, there might still exist $(1+\varepsilon)$-approximation sets of vastly different cardinalities. This motivates to study properties of $(1+\varepsilon)$-approximation sets of minimum cardinality. Moreover, it raises the question whether the minimum possible cardinality can increase for a specific instance when considering partially exact $(1+\varepsilon)$-approximation sets. As shown in~\cite{HBRTV21}, this is indeed the case when considering one-exact $(1+\varepsilon)$-approximation sets: there exist instances for which the minimum cardinality of a one-exact $(1+\varepsilon)$-approximation set is an arbitrary factor larger than the minimum cardinality of an ordinary $(1+\varepsilon)$-approximation set. 

Hence, in this section, we systematically study minimum-cardinality ordinary and partially exact $(1+\varepsilon)$-approximation sets. In Subsection~\ref{subsec:properties}, we first investigate the (weak) efficiency of solutions in minimum-cardinality approximation sets and show that minimum-cardinality ordinary $(1+\varepsilon)$-approximation sets and quasi-$1$-exact $(1+\varepsilon)$-approximation sets can consist of dominated solutions only. For minimum-cardinality one-exact $(1+\varepsilon)$-approximation sets, however, we show that one contained solution must be weakly efficient, but not more than one in general.
Afterwards, in Subsection~\ref{subsec:min-size}, we relate the minimum cardinalities of different types of partially exact approximation sets to the minimum cardinality of an ordinary $(1+\varepsilon)$-approximation set. Here, we show that requiring the $(1+\varepsilon)$-approximation set to be quasi-$1$-exact never changes the minimum cardinality for any instance. In contrast, requiring to be one-exact or quasi-$k$-exact for any $k\geq2$ can increase the minimum cardinality by an arbitrarily large factor.

\subsection{Efficiency of Minimum-Cardinality Approximation Sets}\label{subsec:properties}

In this subsection, we investigate the question of (weak) efficiency of solutions in minimum-cardinality ordinary and partially exact $(1+\varepsilon)$-approximation sets.

\smallskip

Since any dominated solution in any type of partially exact or ordinary $(1+\varepsilon)$-approximation set can always be replaced by an efficient solution dominating it without impairing the obtained approximation guarantee, it follows that there always exist minimum-cardinality $(1+\varepsilon)$-approximation sets of each type that consist of efficient solutions only. In contrast, it is well known that, even in the biobjective case, there are ordinary $(1+\varepsilon)$-approximation sets of minimum cardinality that consist of strictly dominated solutions only. The following proposition generalizes this result to quasi-$1$-exact $(1+\varepsilon)$-approximation sets.

\begin{proposition}\label{prop:approx_dominated}
	For any $\varepsilon> 0$, there is an instance of a biobjective optimization problem that admits a quasi-$1$-exact $(1+\varepsilon)$-approximation set of minimum cardinality consisting of strictly dominated solutions only.
\end{proposition}
\begin{proof}
	The proof is illustrated in Figure~\ref{fig:approx_dominated}.
	Consider a biobjective optimization problem instance consisting of six feasible solutions $x^{1},x^{2},x^{3},x^{4},x^{5},x^{6}$ with
	\begin{align*}
		&f_1(x^{1}) = 1,  &&\quad f_2(x^{1}) = (1+\varepsilon)^2,\\
		&f_1(x^{2}) = 1+\tfrac\varepsilon 2, &&\quad f_2(x^{2}) = (1+\varepsilon) \cdot (1+\tfrac \varepsilon 4),\\
		&f_1(x^{3}) = (1+\varepsilon) \cdot (1+\tfrac \varepsilon 4), &&\quad  f_2(x^{3}) = 1+\tfrac \varepsilon 2,\\
		&f_1(x^{4}) = (1+\varepsilon)^2, &&\quad  f_2(x^{4}) = 1,\\
		&f_1(x^{5}) = 1+\varepsilon, &&\quad  f_2(x^{5}) = (1+\varepsilon) \cdot (1+ \tfrac \varepsilon 2),\\
		&f_1(x^{6}) = (1+\varepsilon) \cdot (1+ \tfrac \varepsilon 2), &&\quad  f_2(x^{6}) =  1+\varepsilon.
	\end{align*}
 
	First, observe that there is no solution that $\edom{}$-dominates the five other solutions, which means that any (quasi-$1$-exact) $(1+\varepsilon)$-approximation set consists of at least two solutions. However, even though $x^{5}$ and~$x^{6}$ are strictly dominated by $x^{2}$ and~$x^{3}$, respectively, the set $\{x^{5},x^{6}\}$ is a quasi-$1$-exact $(1+\varepsilon)$-approximation set:
	$x^{1}$ is $(1+\varepsilon,1)$-approximated by~$x^{5}$, $x^{2}$ is $(1+\varepsilon,1)$-approximated by~$x^{6}$, $x^{3}$ is $(1,1+\varepsilon)$-approximated by~$x^{5}$, and $x^{4}$ is $(1,1+\varepsilon)$-approximated by~$x^{6}$.
\end{proof}

\begin{figure}[ht!]
	\begin{center}
		\includegraphics[]{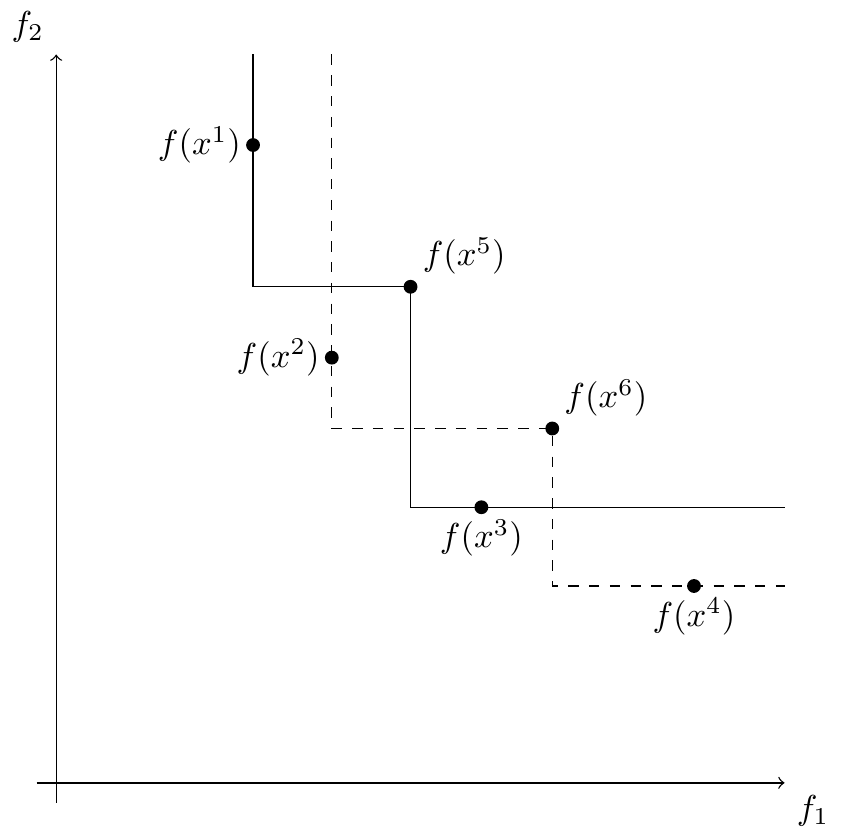}
		\caption[Illustration of the instance constructed in the proof of Proposition~\ref{prop:approx_dominated}]{Illustration of the instance constructed in the proof of Proposition~\ref{prop:approx_dominated}. The set~$\{x^{5},x^{6}\}$ is a minimum-cardinality quasi-$1$-exact $(1+\varepsilon)$-approximation set. The regions that are approximated by~$x^{5}$ and $x^{6}$ are indicated by a solid and a dashed line, respectively.}\label{fig:approx_dominated}
	\end{center}
\end{figure}

For one-exact $(1+\varepsilon)$-approximation sets, the definition implies that each such set must contain at least one weakly efficient solution, namely a solution with minimum $f_1$-value. The following result, however, shows that, even in the biobjective case, there are minimum-cardinality one-exact $(1+\varepsilon)$-approximation sets in which all other solutions are strictly dominated.

\begin{proposition}\label{prop:approx_dominated-one-exact}
    For any $\varepsilon>0$ and any positive integer $n\in\NN_+$, there exists an instance of a biobjective optimization problem that admits a one-exact $(1+\varepsilon)$-approximation set of minimum cardinality~$n+1$ in which all but one of the solutions are strictly dominated.
\end{proposition}

\begin{proof}
    Given~$\varepsilon>0$ and $n\in\NN_+$, let $\delta>0$ so that $(1+\delta)^{2n}=1+\varepsilon$ and consider the following instance of a biobjective optimization problem: The feasible set is given as $X\colonequals\{x^0,\ldots,x^n,\bar{x}^1,\ldots,\bar{x}^n,\tilde{x}_0,\dots,\tilde{x}^n\}$ and
\begin{align*}
    f_1(x^0) & \colonequals 1, & f_2(x^0) & \colonequals (1+\varepsilon)^n, \\
    f_1(x^i) & \colonequals 3i+1, & f_2(x^i) & \colonequals (1+\varepsilon)^{n-i}\cdot(1+\delta)^i, \\
    f_1(\bar{x}^i) & \colonequals 3i, & f_2(\bar{x}^i) & \colonequals (1+\varepsilon)^{n-i}\cdot(1+\delta)^{i-1}, \\
    f_1(\tilde{x}^i) & \colonequals 3i+2, & f_2(\tilde{x}^i) & \colonequals (1+\varepsilon)^{n-i}\cdot\frac{1}{(1+\delta)^i},
\end{align*}
which implies that
\begin{align}
    f_2(\bar{x}^i) & = \frac{1}{1+\delta}\cdot f_2(x^i) = \frac{1}{1+\varepsilon}\cdot f_2(x^{i-1}), \text{ and } \label{eq:barxi} \\
    f_2(\tilde{x}^i) & = \frac{1}{(1+\delta)^{2i}}\cdot f_2(x^i). \label{eq:tildexi}
\end{align}
We now show that $\{x^0,x^1,\ldots,x^n\}$ is a minimum-cardinality one-exact $(1+\varepsilon)$-approximation set, even though~$x^0$ is the only solution in this set that is not strictly dominated. To this end, first note that
\begin{align}
    f_1(x^0) & < f_1(\bar{x}^1) < f_1(x^1) < f_1(\tilde{x}^1) \nonumber \\
    & < f_1(\bar{x}^2) < f_1(x^2) < f_1(\tilde{x}^2) \nonumber \\
    & < \dots \nonumber \\
    & < f_1(\bar{x}^n) < f_1(x^n) < f_1(\tilde{x}^n). \label{eq:chain}
\end{align}
Thus, $x^0$ is not $\edom{1}$-dominated by any other solution, so~$x^0$ must be contained in every one-exact $(1+\varepsilon)$-approximation set. Moreover, the following statements hold: 
\begin{itemize}
    \item $\bar{x}^i$ strictly dominates~$x^i$ for $i\geq1$,
    \item $x^i \edom{1} \tilde{x}^i$ for $i\geq 1$ (by~\eqref{eq:tildexi} and~\eqref{eq:chain}),
    \item $x^i \edom{1} \bar{x}^{i+1}$ (by~\eqref{eq:barxi} and~\eqref{eq:chain}), and, therefore, also $x^i \edom{1} x^{i+1}$ for $0\leq i \leq n-1$,
    \item Neither~$x^i$, nor~$\bar{x}^i$, nor~$\tilde{x}^i$ $\edom{1}$-dominates~$\tilde{x}^{i+1}$ for $0\leq i \leq n-1$ since
    \begin{align*}
        f_2(\tilde{x}^{i+1}) & = \frac{1}{(1+\varepsilon)\cdot(1+\delta)}\cdot f_2(\tilde{x}^i) \\
        & < \frac{1}{1+\varepsilon}\cdot f_2(\tilde{x}^i) < \frac{1}{1+\varepsilon}\cdot f_2(\bar{x}^i) < \frac{1}{1+\varepsilon}\cdot f_2(x^i).
    \end{align*}
    \item Neither~$x^i$, nor~$\bar{x}^i$, nor~$\tilde{x}^i$ $\edom{1}$-dominates~$\tilde{x}^j$ for $0\leq j<i\leq n$ (by~\eqref{eq:chain})
     \item Neither~$x^i$, nor~$\bar{x}^i$, nor~$\tilde{x}^i$ $\edom{1}$-dominates~$\tilde{x}^j$ for $i+2\leq j \leq n$ since
    \begin{align*}
        f_2(\tilde{x}^j) & < f_2(\tilde{x}^{i+1})
        < \frac{1}{1+\varepsilon}\cdot f_2(\tilde{x}^i) < \frac{1}{1+\varepsilon}\cdot f_2(\bar{x}^i) < \frac{1}{1+\varepsilon}\cdot f_2(x^i).
    \end{align*}
    In total, this implies that $\{x^0,x^1,\ldots,x^n\}$ is a one-exact $(1+\varepsilon)$-approximation set, even though~$x^0$ is the only solution in this set that is not strictly dominated.
    
    Moreover, for each~$i\in\{1,\dots,n\}$, in order to $\edom{1}$-dominate~$\tilde{x}^i$, at least one of the three solutions~$x^i$, $\bar{x}^i$, and $\tilde{x}^i$ must be contained in any one-exact $(1+\varepsilon)$-approximation set. Consequently, any one-exact $(1+\varepsilon)$-approximation set has cardinality at least~$n+1$.\qedhere
\end{itemize}
\end{proof}

The final result of this subsection, which will be used when analyzing the minimum cardinality of quasi-$1$-exact $(1+\varepsilon)$-approximation sets in the following subsection, states that any $(1+\varepsilon)$-approximation set consisting only of weakly efficient solutions must actually be quasi-$1$-exact.

\begin{proposition}\label{prop:weakly-efficient}
	Any $(1+\varepsilon)$-approximation set in which all solutions are weakly efficient is a quasi-$1$-exact $(1+\varepsilon)$-approximation set.
\end{proposition}

\begin{proof}
	Let~$P$ be a $(1+\varepsilon)$-approximation set in which all solutions are weakly efficient. Let~$x' \in X$ be an arbitrary feasible solution and let~$x \in P$ be a solution that $\edom{}$-dominates~$x'$. If $x'$ is not~$\edom{(1)}$-dominated by~$x$, i.e., we do not have~$f_i(x) \leq f_i(x')$ for any~$i \in \{1,\ldots,p\}$, we immediately obtain that~$x$ is strictly dominated by~$x'$, which contradicts the weak efficiency of~$x$.
\end{proof}

\subsection{Relations Between Minimum Cardinalities}\label{subsec:min-size}

We now relate the minimum cardinalities of different types of partially exact approximation sets to the minimum cardinality of an ordinary $(1+\varepsilon)$-approximation set.

\smallskip

Concerning one-exact $(1+\varepsilon)$-approximation sets, it is shown in~\cite{HBRTV21} that, even for biobjective problems, the minimum cardinality of a such a set can be an arbitrary factor larger than the minimum cardinality of an ordinary $(1+\varepsilon)$-approximation set for certain instances. In contrast to this, we now use Proposition~\ref{prop:weakly-efficient} to show that requiring a quasi-$1$-exact $(1+\varepsilon)$-approximation set instead of only an ordinary $(1+\varepsilon)$-approximation set never changes the minimum cardinality for any instance.

\begin{theorem}\label{thm:min-size-quasi-1-exact}
    For any instance of a multiobjective optimization problem and any~$\varepsilon > 0$, there exists a minimum-cardinality ordinary $(1+\varepsilon)$-approximation set that is quasi-$1$-exact. In particular, the minimum cardinalities of an ordinary $(1+\varepsilon)$-approximation set and of a quasi-$1$-exact $(1+\varepsilon)$-approximation set coincide for every instance.
\end{theorem}

\begin{proof}
	Let $P^*$ be a smallest ordinary $(1+\varepsilon)$-approximation set. Turn $P^{*}$ into a quasi-$1$-exact $(1+\varepsilon)$-approximation set~$P'$ of the same cardinality that consists of weakly efficient solutions only by replacing any strictly dominated solution~$x\in P^{*}$ by a weakly efficient solution that strictly dominates it (this does not impair the approximation guarantee). By Proposition~\ref{prop:weakly-efficient}, the set~$P'$ is then indeed a quasi-$1$-exact $(1+\varepsilon)$-approximation set.
\end{proof}


Contrasting the result on the minimum cardinality of quasi-$1$-exact $(1+\varepsilon)$-approximation sets shown in Theorem~\ref{thm:min-size-quasi-1-exact}, we now show that, for $k\geq 2$, the minimum cardinality of a quasi-$k$-exact $(1+\varepsilon)$-approximation set can be an arbitrary factor larger than the minimum cardinality of an ordinary $(1+\varepsilon)$-approximation set - even in the case of three objectives ($p=3$). The proof of the following theorem generalizes the construction used in the proof of Theorem~2 in~\cite{HBRTV21}.

\begin{theorem}\label{thm:min-size-quasi-2-exact}
For any $\varepsilon>0$ and any positive integer $n\in\NN_+$, there exist instances of $3$-objective optimization problems such that $\lvert P^{(2),*}\rvert>n\cdot\lvert P^*\rvert$, where~$P^{(2),*}$ denotes a minimum-cardinality quasi-$2$-exact $(1+\varepsilon)$-approximation set and~$P^*$ denotes a minimum-cardinality ordinary $(1+\varepsilon)$-approximation set.
\end{theorem}

\begin{proof}
Given~$\varepsilon>0$ and $n\in\NN_+$, we construct an instance of a $3$-objective optimization problem with $\lvert P^{(2),*}\rvert=n+1$ and $\lvert P^*\rvert=1$ as follows: The feasible set is given as $X\colonequals\{x^0,\dots,x^n\}$ and, for $j=0,\dots,n$, we define $f(x^j) = \left(f_1(x^j),f_2(x^j),f_3(x^j)\right)$ via
\begin{align*}
    f(x^j) \colonequals \left(1+\frac{n-j}{n}\cdot\varepsilon,\, 1+\frac{n-j}{n}\cdot\varepsilon,\, (1+\varepsilon)^{2j+1}\right).
\end{align*}
Note that this implies that
\begin{align*}
    f_i(x^0) > f_i(x^1) > \dots > f_i(x^n) \quad \text{ for } i=1,2,
\end{align*}
but
\begin{align*}
    f_3(x^j) = \frac{1}{(1+\varepsilon)^2}\cdot f_3(x^{j+1}) < \frac{1}{1+\varepsilon}\cdot f_3(x^{j+1})
\end{align*}
for $j=0,\dots,n-1$. Consequently, no solution $\edom{(2)}$-dominates any other solution, which implies that $P^{(2),*}=X$ with cardinality~$n+1$. The solution~$x^0$ with $f(x^0)=\left(1+\varepsilon,1+\varepsilon,1+\varepsilon\right)$, however, $\edom{}$-dominates all other solutions, so $\{x^0\}$ is a $(1+\varepsilon)$-approximation set.
\end{proof}




\section{Computation of Approximation Sets}\label{sec:computation}

After studying the existence of (fully) polynomial-cardinality approximation sets and the impact of using partially exact approximate dominance relations on the minimum cardinality of approximation sets, we now consider the computability of such sets.

\smallskip

As shown in~\cite{Papadimitriou+Yannakakis:multicrit-approx}, a $(1+\varepsilon)$-approximation set can be computed in (fully) polynomial time for every $\varepsilon>0$ if and only if the auxiliary problem~$\gap_\delta(b)$ can be solved in (fully) polynomial time for every vector~$b$ and every~$\delta>0$. Hence, a natural question is whether (fully) polynomial-time solvability of~$\gap_\delta(b)$ is also sufficient for the (fully) polynomial-time computability of partially exact $(1+\varepsilon)$-approximation sets. The following theorem gives a negative answer to this question even for biobjective problems and the least-demanding type of partially exact $(1+\varepsilon)$-approximation sets. The proof is partly similar to the proof of Theorem~$1$ in~\cite{Vassilvitskii+Yannakakis:trade-off-curves}, where it is shown that polynomial-time algorithms only generating solutions by solving $\gap_\delta(b)$ cannot approximate the minimum cardinality of an ordinary $(1+\varepsilon)$-approximation set to a factor better than~$3$ in the biobjective case (such algorithms are called \emph{generic} in~\cite{Vassilvitskii+Yannakakis:trade-off-curves}).


\begin{theorem}\label{thm:gap-not-enough}
    There exists no polynomial-time algorithm that returns a quasi-$1$-exact $(1+\varepsilon)$-approximation set for every biobjective problem instance and every $\varepsilon>0$, and generates feasible solutions only by solving $\gap_\delta(b)$ for points~$b$ of polynomial encoding size and polynomial values of~$\frac{1}{\delta}$. 
\end{theorem}

\begin{proof}
    Given some (rational) $\varepsilon>0$, consider two instances~$I_1,I_2$ with feasible sets~$X^{I_1}=\{x^1\}$ and~$X^{I_2}=\{x^1,x^2\}$ and the objective function values given by
    \begin{align*}
        f(x^1) = \left(1+\frac{1}{l(\varepsilon)},\,1+\frac{1}{l(\varepsilon)}\right), \quad f(x^2) = \left(1,\,1\right),
    \end{align*}
    where~$l(\varepsilon)$ is a positive integer that is exponential in the input size and in~$\frac{1}{\varepsilon}$.
    We show that no algorithm that generates feasible solutions only by solving $\gap_\delta(b)$ for values $\delta\geq\frac{1}{l(\varepsilon)}$ can distinguish between the two instances (i.e., detect whether~$x^2$ is part of the feasible set). Since any smaller value of~$\delta$ would be exponential in the input size and in~$\frac{1}{\varepsilon}$ by definition of~$l(\varepsilon)$, and any quasi-$1$-exact $(1+\varepsilon)$-approximation set for instance~$I_2$ must include~$x^2$, this will show the claim.

    \smallskip

    So consider a call of $\gap_\delta(b)$ for some point~$b$ and some $\delta\geq\frac{1}{l(\varepsilon)}$. We distinguish two cases in order to show that $\gap_\delta(b)$ can always return either~$x^1$ or NO as a correct answer for both instance~$I_1$ and instance~$I_2$. If $b_i\geq f_i(x^1)$ for $i=1,2$, then $\gap_\delta(b)$ can return~$x^1$ for both instances. Otherwise, we have $b_j < f_j(x^1) = 1+\frac{1}{l(\varepsilon)}$ for some~$j\in\{1,2\}$. Using that $\delta\geq\frac{1}{l(\varepsilon)}$, this implies that
    \begin{align*}
        f_j(x^2) = 1 \geq \frac{1+\frac{1}{l(\varepsilon)}}{1+\delta} > \frac{b_j}{1+\delta}.
    \end{align*}
    Consequently, for both instances, there exists no solution with $j$-th objective value less than or equal to~$\frac{b_j}{1+\delta}$, so $\gap_\delta(b)$ can return NO. This shows that an algorithm that generates feasible solutions only by solving $\gap_\delta(b)$ for values of $\delta\geq\frac{1}{l(\varepsilon)}$ cannot detect whether~$x^2$ is part of the feasible set as claimed.
\end{proof}

Theorem~\ref{thm:gap-not-enough} shows that harder-to-solve auxiliary problems than the gap problem must be used for computing partially exact $(1+\varepsilon)$-approximation sets in polynomial time.
We now use Proposition~\ref{prop:weakly-efficient} to show that, for biobjective problems, several known algorithms for computing (small-cardinality) $(1+\varepsilon)$-approximation sets based on such auxiliary problems actually yield quasi-$1$-exact $(1+\varepsilon)$-approximation sets.



\begin{theorem}\label{thm:computation--quasi-1-exact}
\begin{enumerate}
    \item[1)] For biobjective optimization problems for which $\constrained^1$ and $\constrained^2$ are solvable in (fully) polynomial time, a quasi-$1$-exact $(1+\varepsilon)$-approximation set with minimum cardinality can be computed in (fully) polynomial time for every instance.
    \item[2)] For biobjective optimization problems for which $\dualrestrict^1_\delta$ or $\dualrestrict^2_\delta$ is solvable in (fully) polynomial time, a quasi-$1$-exact $(1+\varepsilon)$-approximation set with cardinality at most twice the cardinality of a smallest quasi-$1$-exact $(1+\varepsilon)$-approximation set can be computed in (fully) polynomial time for every instance.
\end{enumerate}
\end{theorem}

\begin{proof}
    \begin{enumerate}
    \item[1)] If $\constrained^1$ and $\constrained^2$ are solvable in (fully) polynomial time, an ordinary $(1+\varepsilon)$-approximation set with with minimum cardinality can be computed in (fully) polynomial time for every instance by a greedy procedure as shown in~\cite{Diakonikolas+Yannakakis:approx-pareto-sets}. Since this set contains only weakly efficient solutions~\cite{Bazgan+etal:representation}, it is actually a quasi-$1$-exact $(1+\varepsilon)$-approximation set of minimum cardinality by Proposition~\ref{prop:weakly-efficient}.
    \item[2)] If $\dualrestrict^1_\delta$ or $\dualrestrict^2_\delta$ is solvable in (fully) polynomial time, an ordinary $(1+\varepsilon)$-approximation set with cardinality at most twice the cardinality of a smallest ordinary $(1+\varepsilon)$-approximation set (which equals the minimum cardinality of a quasi-$1$-exact $(1+\varepsilon)$-approximation set by Theorem~\ref{thm:min-size-quasi-1-exact})
    can be computed in (fully) polynomial time for every instance as shown in~\cite{Diakonikolas+Yannakakis:approx-pareto-sets,Bazgan+etal:representation}. Moreover, the $(1+\varepsilon)$-approximation sets returned by the algorithms in~\cite{Diakonikolas+Yannakakis:approx-pareto-sets,Bazgan+etal:representation} contain only weakly efficient solutions~\cite{Bazgan+etal:representation}. Consequently, by Proposition~\ref{prop:weakly-efficient}, these sets are actually quasi-$1$-exact $(1+\varepsilon)$-approximation sets.\qedhere
\end{enumerate}
\end{proof}

We remark that, as shown in~\cite{HBRTV21}, the same results can be achieved even for one-exact $(1+\varepsilon)$-approximation sets, but modified algorithms targeted specifically at the computation of one-exact $(1+\varepsilon)$-approximation sets are necessary in this case. Moreover, it has been shown in~\cite{HBRTV21} that, for general $p$-objective problems, one-exact $(1+\varepsilon)$-approximation sets can be computed in (fully) polynomial time if and only if $\dualrestrict^1_{\delta}(b_2,\ldots, b_p)$ can be solved for any choice of the bounds $b_2,\ldots,b_p$ and any $\delta>0$ in (fully) polynomial time.

\smallskip

We finish this section by considering the case where all feasible solutions are given explicitly in the input. For this case, we now present two related approaches for the polynomial-time computation of $R$-approximation sets for approximate dominance relations~$R$ as in Theorem~\ref{thm:generalexistence} under the assumption that checking whether a given solution~$x\in X$ $R'$-dominates another given solution~$x'\in X$ (for~$R'$ as in the theorem) is possible in polynomial time. This class of relations, in particular, includes quasi-$k$-exact $(1+\varepsilon)$-dominance for any $k \leq \lceil \frac{p}{2}\rceil$ as well as one-exact $(1+\varepsilon)$-dominance and ordinary $(1+\varepsilon)$-dominance.

The first approach is based directly on the proof of Theorem~\ref{thm:generalexistence}. Given that checking whether a given solution~$x\in X$ $R'$-dominates another given solution~$x'\in X$ (for~$R'$ as in the theorem) is possible in polynomial time, we can generate the directed graph corresponding to the approximate dominance relation~$R'$ in each weakly nondominated nonempty hyperrectangle considered in the proof in polynomial time and compute a dominating set in each of these graphs. While computing a minimum-cardinality dominating set is $\textsf{NP}$-hard in general, a simple greedy algorithm can be used to compute a dominating set that is larger by at most a factor~$1+\log\lvert V \rvert$, where~$V$ denotes the vertex set~\cite{Johnson74a}. This yields a polynomial-cardinality set in our situation since~$\lvert V \rvert$ is bounded by~$\lvert X \rvert$ in all cases, which is polynomial in the input size since~$X$ is given explicitly in the input.


A second related approach, which additionally yields a bound on the cardinality of the obtained $R$-approximation set, consists of constructing the directed graph on~$X$ corresponding to the approximate dominance relation~$R$ (which is possible in polynomial time given that $R'$-dominance of solutions can be checked in polynomial time) and to compute a dominating set of cardinality at most~$1+\log\lvert X \rvert$ times the minimum cardinality of dominating set in this graph using the algorithm from~\cite{Johnson74a}. The resulting set then corresponds to an $R$-approximation set of cardinality at most~$1+\log\lvert X \rvert$ times the minimum cardinality of an $R$-approximation set.

\section{Conclusion}

In this paper, we explore the borderline between tractability and intractability when approximating multiobjective optimization problems using more demanding approximate dominance relations than $(1+\varepsilon)$-dominance. We show that, under very weak assumptions, there always exist $(1+\varepsilon)$-approximation sets of fully polynomial cardinality such that every feasible solution is, in fact, $1$-approximated with respect to at least half of the objectives functions if we allow these objective functions to differ between different feasible solutions. If a polynomial-cardinality approximation set is required to be exact in one specified objective function, however, exactness cannot be ensured in a second objective function in general, even if this second objective function is allowed to differ depending on the approximated solution. This leads to the ``frontier of intractability'' being located in between one-exact $(1+\varepsilon)$-approximation sets and one-exact, quasi-$2$-exact $(1+\varepsilon)$-approximation sets and in between quasi-$\lceil \frac{p}{2}\rceil$-exact $(1+\varepsilon)$-approximation sets and quasi-$(\lceil \frac{p}{2} \rceil+1)$-exact $(1+\varepsilon)$-approximation sets.

While first positive and negative results concerning the polynomial-time computability of quasi-$k$-exact $(1+\varepsilon)$-approximation sets have been obtained in Section~\ref{sec:computation}, an important question for future research is to obtain further computability results that extend beyond the biobjective case or the situation where the feasible set is given explicitly in the input. Specifically, since both the polynomial-time computability of ordinary $(1+\varepsilon)$-approximation sets and of one-exact $(1+\varepsilon)$-approximation sets have been characterized by the polynomial-time solvability of certain auxiliary problems ($\gap$ and $\dualrestrict$, respectively), it would be interesting to investigate whether a similar characterization via an auxiliary problem can also be obtained for quasi-$k$-exact $(1+\varepsilon)$-approximation sets for $k \leq \lceil \frac{p}{2}\rceil$. As our results shown, such an auxiliary problem would have to be strictly harder to solve than $\gap$.

\section*{Acknowledgements}

This work was supported by the bilateral cooperation project ``Approximation methods for multiobjective optimization problems'' funded by the German Academic Exchange Service (DAAD, Project-ID 57388848) and by Campus France, PHC PROCOPE 2018 (Project no. 40407WF) as well as by the Deutsche Forschungsgemeinschaft (DFG, German Research Foundation) – Project number~398572517.




\bibliography{literature}


\begin{thebibliography}{15}
\ifx \bisbn   \undefined \def \bisbn  #1{ISBN #1}\fi
\ifx \binits  \undefined \def \binits#1{#1}\fi
\ifx \bauthor  \undefined \def \bauthor#1{#1}\fi
\ifx \batitle  \undefined \def \batitle#1{#1}\fi
\ifx \bjtitle  \undefined \def \bjtitle#1{#1}\fi
\ifx \bvolume  \undefined \def \bvolume#1{\textbf{#1}}\fi
\ifx \byear  \undefined \def \byear#1{#1}\fi
\ifx \bissue  \undefined \def \bissue#1{#1}\fi
\ifx \bfpage  \undefined \def \bfpage#1{#1}\fi
\ifx \blpage  \undefined \def \blpage #1{#1}\fi
\ifx \burl  \undefined \def \burl#1{\textsf{#1}}\fi
\ifx \doiurl  \undefined \def \doiurl#1{\url{https://doi.org/#1}}\fi
\ifx \betal  \undefined \def \betal{\textit{et al.}}\fi
\ifx \binstitute  \undefined \def \binstitute#1{#1}\fi
\ifx \binstitutionaled  \undefined \def \binstitutionaled#1{#1}\fi
\ifx \bctitle  \undefined \def \bctitle#1{#1}\fi
\ifx \beditor  \undefined \def \beditor#1{#1}\fi
\ifx \bpublisher  \undefined \def \bpublisher#1{#1}\fi
\ifx \bbtitle  \undefined \def \bbtitle#1{#1}\fi
\ifx \bedition  \undefined \def \bedition#1{#1}\fi
\ifx \bseriesno  \undefined \def \bseriesno#1{#1}\fi
\ifx \blocation  \undefined \def \blocation#1{#1}\fi
\ifx \bsertitle  \undefined \def \bsertitle#1{#1}\fi
\ifx \bsnm \undefined \def \bsnm#1{#1}\fi
\ifx \bsuffix \undefined \def \bsuffix#1{#1}\fi
\ifx \bparticle \undefined \def \bparticle#1{#1}\fi
\ifx \barticle \undefined \def \barticle#1{#1}\fi
\bibcommenthead
\ifx \bconfdate \undefined \def \bconfdate #1{#1}\fi
\ifx \botherref \undefined \def \botherref #1{#1}\fi
\ifx \url \undefined \def \url#1{\textsf{#1}}\fi
\ifx \bchapter \undefined \def \bchapter#1{#1}\fi
\ifx \bbook \undefined \def \bbook#1{#1}\fi
\ifx \bcomment \undefined \def \bcomment#1{#1}\fi
\ifx \oauthor \undefined \def \oauthor#1{#1}\fi
\ifx \citeauthoryear \undefined \def \citeauthoryear#1{#1}\fi
\ifx \endbibitem  \undefined \def \endbibitem {}\fi
\ifx \bconflocation  \undefined \def \bconflocation#1{#1}\fi
\ifx \arxivurl  \undefined \def \arxivurl#1{\textsf{#1}}\fi
\csname PreBibitemsHook\endcsname

\bibitem{Ehrgott:hard-to-say}
\begin{bchapter}
\bauthor{\bsnm{Ehrgott}, \binits{M.}}:
\bctitle{Hard to say it's easy - four reasons why combinatorial multiobjective
  programmes are hard}.
In: \bbtitle{Proceedings of the 14th International Conference on Multiple
  Criteria Decision Making ({MCDM})}.
\bsertitle{Lecture Notes in Economics and Mathematical Systems},
vol. \bseriesno{487},
pp. \bfpage{69}--\blpage{80}
(\byear{1998})
\end{bchapter}
\endbibitem

\bibitem{Safer:PHD}
\begin{botherref}
\oauthor{\bsnm{Safer}, \binits{H.M.}}:
Fast approximation schemes for multi-criteria combinatorial optimization.
PhD thesis,
Massachusetts Institute of Technology
(1992)
\end{botherref}
\endbibitem

\bibitem{Safer+Orlin:combinatorial}
\begin{botherref}
\oauthor{\bsnm{Safer}, \binits{H.M.}},
\oauthor{\bsnm{Orlin}, \binits{J.B.}}:
Fast Approximation Schemes for Multi-criteria Combinatorial Optimization.
Working Paper 3756-95, Sloan School of Management, Massachusetts Institute of
  Technology
(1995)
\end{botherref}
\endbibitem

\bibitem{Papadimitriou+Yannakakis:multicrit-approx}
\begin{bchapter}
\bauthor{\bsnm{Papadimitriou}, \binits{C.}},
\bauthor{\bsnm{Yannakakis}, \binits{M.}}:
\bctitle{On the approximability of trade-offs and optimal access of web
  sources}.
In: \bbtitle{Proceedings of the 41st Annual {IEEE} Symposium on the Foundations
  of Computer Science ({FOCS})},
pp. \bfpage{86}--\blpage{92}
(\byear{2000})
\end{bchapter}
\endbibitem

\bibitem{Bazgan+etal:min-pareto}
\begin{barticle}
\bauthor{\bsnm{Bazgan}, \binits{C.}},
\bauthor{\bsnm{Jamain}, \binits{F.}},
\bauthor{\bsnm{Vanderpooten}, \binits{D.}}:
\batitle{Approximate {P}areto sets of minimal size for multi-objective
  optimization problems}.
\bjtitle{Operations Research Letters}
\bvolume{43}(\bissue{1}),
\bfpage{1}--\blpage{6}
(\byear{2015})
\end{barticle}
\endbibitem

\bibitem{Diakonikolas+Yannakakis:approx-pareto-sets}
\begin{barticle}
\bauthor{\bsnm{Diakonikolas}, \binits{I.}},
\bauthor{\bsnm{Yannakakis}, \binits{M.}}:
\batitle{Small approximate {P}areto sets for biobjective shortest paths and
  other problems}.
\bjtitle{{SIAM} Journal on Computing}
\bvolume{39}(\bissue{4}),
\bfpage{1340}--\blpage{1371}
(\byear{2009})
\end{barticle}
\endbibitem

\bibitem{Koltun+Papadimitriou:approx-dom-repr}
\begin{barticle}
\bauthor{\bsnm{Koltun}, \binits{V.}},
\bauthor{\bsnm{Papadimitriou}, \binits{C.}}:
\batitle{Approximately dominating representatives}.
\bjtitle{Theoretical Computer Science}
\bvolume{371}(\bissue{3}),
\bfpage{148}--\blpage{154}
(\byear{2007})
\end{barticle}
\endbibitem

\bibitem{Vassilvitskii+Yannakakis:trade-off-curves}
\begin{barticle}
\bauthor{\bsnm{Vassilvitskii}, \binits{S.}},
\bauthor{\bsnm{Yannakakis}, \binits{M.}}:
\batitle{Efficiently computing succinct trade-off curves}.
\bjtitle{Theoretical Computer Science}
\bvolume{348}(\bissue{2-3}),
\bfpage{334}--\blpage{356}
(\byear{2005})
\end{barticle}
\endbibitem

\bibitem{HBRTV21}
\begin{barticle}
\bauthor{\bsnm{Herzel}, \binits{A.}},
\bauthor{\bsnm{Bazgan}, \binits{C.}},
\bauthor{\bsnm{Ruzika}, \binits{S.}},
\bauthor{\bsnm{Thielen}, \binits{C.}},
\bauthor{\bsnm{Vanderpooten}, \binits{D.}}:
\batitle{One-exact approximate pareto sets}.
\bjtitle{Journal of Global Optimization}
\bvolume{80}(\bissue{1}),
\bfpage{87}--\blpage{115}
(\byear{2021})
\end{barticle}
\endbibitem

\bibitem{Herzel+etal:survey}
\begin{barticle}
\bauthor{\bsnm{Herzel}, \binits{A.}},
\bauthor{\bsnm{Ruzika}, \binits{S.}},
\bauthor{\bsnm{Thielen}, \binits{C.}}:
\batitle{Approximation methods for multiobjective optimization problems: A
  survey}.
\bjtitle{{INFORMS} Journal on Computing}
\bvolume{33}(\bissue{4}),
\bfpage{1284}--\blpage{1299}
(\byear{2021})
\end{barticle}
\endbibitem

\bibitem{ABKKW06}
\begin{barticle}
\bauthor{\bsnm{Alon}, \binits{N.}},
\bauthor{\bsnm{Brightwell}, \binits{G.R.}},
\bauthor{\bsnm{Kierstead}, \binits{H.A.}},
\bauthor{\bsnm{Kostochka}, \binits{A.V.}},
\bauthor{\bsnm{Winkler}, \binits{P.}}:
\batitle{Dominating sets in k-majority tournaments}.
\bjtitle{Journal of Combinatorial Theory, Series {B}}
\bvolume{96}(\bissue{3}),
\bfpage{374}--\blpage{387}
(\byear{2006})
\end{barticle}
\endbibitem

\bibitem{Aissi+etal:MinCut}
\begin{barticle}
\bauthor{\bsnm{Aissi}, \binits{H.}},
\bauthor{\bsnm{Mahjoub}, \binits{A.R.}},
\bauthor{\bsnm{McCormick}, \binits{S.T.}},
\bauthor{\bsnm{Queyranne}, \binits{M.}}:
\batitle{Strongly polynomial bounds for multiobjective and parametric global
  minimum cuts in graphs and hypergraphs}.
\bjtitle{Mathematical Programmming}
\bvolume{154}(\bissue{1-2}),
\bfpage{3}--\blpage{28}
(\byear{2015})
\end{barticle}
\endbibitem

\bibitem{Papadimitriou+Yannakakis:complexity}
\begin{barticle}
\bauthor{\bsnm{Papadimitriou}, \binits{C.}},
\bauthor{\bsnm{Yannakakis}, \binits{M.}}:
\batitle{Optimization, approximation, and complexity classes}.
\bjtitle{Journal of Computer and System Sciences}
\bvolume{43}(\bissue{3}),
\bfpage{425}--\blpage{440}
(\byear{1991})
\end{barticle}
\endbibitem

\bibitem{Bazgan+etal:representation}
\begin{barticle}
\bauthor{\bsnm{Bazgan}, \binits{C.}},
\bauthor{\bsnm{Jamain}, \binits{F.}},
\bauthor{\bsnm{Vanderpooten}, \binits{D.}}:
\batitle{Discrete representation of the non-dominated set for multi-objective
  optimization problems using kernels}.
\bjtitle{European Journal of Operational Research}
\bvolume{260}(\bissue{3}),
\bfpage{814}--\blpage{827}
(\byear{2017})
\end{barticle}
\endbibitem

\bibitem{Johnson74a}
\begin{barticle}
\bauthor{\bsnm{Johnson}, \binits{D.S.}}:
\batitle{Approximation algorithms for combinatorial problems}.
\bjtitle{Journal of Computer and System Sciences}
\bvolume{9}(\bissue{3}),
\bfpage{256}--\blpage{278}
(\byear{1974})
\end{barticle}
\endbibitem

\end{thebibliography}
 
\end{document}